\documentclass[11pt,leqno]{amsart}
\usepackage{amsmath,amssymb,amsthm}
\usepackage{hyperref}

\usepackage{tikz-cd}

\usepackage{hyperref}

\usepackage{bbm}

\DeclareMathOperator{\diam}{diam\,}
\DeclareMathOperator{\co}{co}

\renewcommand{\geq}{\geqslant}
\renewcommand{\leq}{\leqslant}

\newcommand{\innt}{\operatorname{int}}

\newcommand{\supp}{\operatorname{supp}}
\newcommand{\spann}{\operatorname{span}}

\newtheorem{theorem}{Theorem}[section]
\newtheorem{lemma}[theorem]{Lemma}

\newtheorem{proposition}[theorem]{Proposition}

\theoremstyle{definition}

\newtheorem{example}[theorem]{Example}

\theoremstyle{remark}
\newtheorem{remark}[theorem]{Remark}
\newtheorem{question}{Question}
\numberwithin{equation}{section}

\def\fnote#1{\footnote}

\def\ignora#1{}
\def\n3#1{\left\vert  \! \left\vert \! \left\vert \, #1 \, \right\vert \!
  \right\vert \! \right\vert }

\renewcommand{\leq}{\le}

\usepackage{accents}
\usepackage[most]{tcolorbox}
\usepackage{comment}

\let\emptyset\varnothing

\newcommand{\eps}{\varepsilon}
\newcommand{\conv}{\text{co}}

\begin{document}


\author{ Gin\'es L\'opez-P\'erez }\address{Universidad de Granada, Facultad de Ciencias. Departamento de An\'{a}lisis Matem\'{a}tico, 18071-Granada
(Spain)} \email{ glopezp@ugr.es}


\author{ Abraham Rueda Zoca }\address{Universidad de Granada, Facultad de Ciencias. Departamento de An\'{a}lisis Matem\'{a}tico, 18071-Granada
(Spain)} \email{ abrahamrueda@ugr.es}
\urladdr{\url{https://arzenglish.wordpress.com}}

\subjclass[2020]{46B03, 46B04, 46B20, 46B22}

\keywords{Radius; diameter; strongly regular; Radon-Nikodym}

\title{Strongly regular Banach spaces with big weakly open subsets in the unit ball}

\begin{abstract}
We construct, given $1<p<\infty$, a Banach space $Y$ and a closed, convex and symmetric set $L\subseteq B_Y$ with the following properties:
\begin{enumerate}
\item $Y^{**}$ is strongly regular (henceforth, $Y$ is strongly regular). 
\item Every non-empty relatively weakly-star open subset of $\overline{L}^{w^*}$ (the $w^*$ closure of $L$ in $Y^{**}$) has radius one. In particular, every non-empty relatively weakly open subset of $L$ has radius $1$.
\item Every non-empty relatively weakly open subset of $L$ has diameter, at least, $2^\frac{1}{p}$. 
\end{enumerate}
This constitutes an advance to the question whether there exists a strongly regular Banach spaces satisfying that every non-empty relatively weakly open subset of the unit ball has radius $1$. As a partial answer, we get that for every $\varepsilon>0$ there exists a strongly regular Banach spaces where weakly open subsets have radius, at least, $1-\varepsilon$.
\end{abstract}

\maketitle

\markboth{GIN\'ES L\'OPEZ-P\'EREZ AND ABRAHAM RUEDA ZOCA}{STRONGLY REGULAR BANACH SPACES WITH BIG WEAKLY OPEN ...}

\section{Introduction}

It is fair to say that one of the most studied properties in Banach spaces theory is the Radon-Nikod\'{y}m property, from now on RNP, which was origillay defined by the validity of a vector version of the classic Radon-Nikod\'{y}m theorem on derivation of measures. Namely, $X$ has the RNP if for any $\sigma$-finite measure space $(\Omega, \Sigma, \mu)$ and any $\mu$-continuous vector measure of bounded variation $\nu\colon \Sigma \rightarrow X$, there exists a Bochner integrable function $f\colon \Omega \rightarrow X$ such that 
\begin{equation}\label{eq:deriv} \nu(A) =\int_A f \, d\mu \end{equation}
for every $A \in \Sigma$, see \cite{Bourgin} or \cite{disuhl} for details.

Apart from the very definition, which is the natural desire of the extension of the famous Radon-Nikod\'{y}m theorem to an infinite-dimensional context, there is big amount of characterisations of the RNP in multiple different contexts (as a sample of this, we refer the reader to \cite[Section VII.6]{disuhl} where more than 25 characterisations of the RNP are exhibited). From the point of view of the geometry of Banach spaces, probably the most famous characterisation of the RNP is given in terms of \textit{dentability}: a Banach space $X$ has the RNP if, and only if, every bounded and subset $C$ of $X$ contains slices (i.e. intersections of open half spaces with $C$) of arbitrarily small diameter (the above is shortened saying that \textit{$C$ is dentable}).

Other three properties closely related to the RNP were intensively studied in the eighties by many authors. In order to shed light on the immediate relations among them, let us introduce the following notation: for a bounded set $C$ of a Banach space $X$, a \textit{slice of $C$} is, as we indicated before, the intersection of an open half space with $C$; in other words, a set of the following form
$$S(C,f,\alpha):=\{x\in C: f(x)>\sup f(C)-\alpha\}$$
where $f\in X^*$ and $\alpha>0$. If $C$ is additionally convex, by a \textit{convex combination of slices of $C$} we mean a set of the following form
$$\sum_{i=1}^n \lambda_i S_i$$
where $\lambda_1,\ldots, \lambda_n\in [0,1]$ are such that $\sum_{i=1}^n \lambda_i=1$ and every $S_i$ is a slice of $C$.

Coming back to properties related to the RNP, given a Banach space $X$, we say that $X$:
\begin{enumerate}
    \item has the RNP if every bounded subset $C$ of $X$ is dentable (i.e. $C$ contains slices of $C$ of arbitrarily small diameter);
    \item has the \textit{point of continuity property (PCP)} if every bounded subset $C$ of $X$ contains non-empty weakly open subsets of $C$ of arbitrarily small diameter;
    \item has the \textit{convex point of continuity property (CPCP)} if every bounded and convex subset $C$ of $X$  contains non-empty weakly open subsets of $C$ of arbitrarily small diameter;
    \item is \textit{strongly regular (SR)} if every bounded, closed and convex subset $C$ of $X$  contains convex combinations of slices of $C$ of arbitrarily small diameter.
\end{enumerate}
We refer the reader to \cite{boro,ggms,gmscouT,gms} for background around the above properties. Since every slice is a weak open set, it is clear that the RNP implies the PCP, whereas the PCP obviously implies the CPCP. It is not trivial that the CPCP implies SR, but this result follows from the fact that, given any closed, convex and bounded set $C$ in a Banach space $X$, every non-empty weakly open subset of $C$ contains a convex combination of slices of $C$ \cite[Lemma II.1]{ggms}. To sum up, we have that RNP$\Rightarrow$PCP$\Rightarrow$CPCP$\Rightarrow$SR, and it is known that none of the above converse does hold (see \cite{boro}, \cite{gmscouT} and \cite{gms} for respective counterexamples). 

Since all the above properties are preserved under equivalent renorming natural questions appear at this point. Let us focus on the RNP: in view of the previous characterisation, given a Banach space $X$ failing the RNP, we can find a bounded subset $C$ of $X$ and a positive $\varepsilon\leq \diam(C)$ such that every slice of $C$ has diameter, at least, $\varepsilon$. In order to push further this characterisation of the failure of the RNP, two natural questions arise:
\begin{itemize}
\item[i)] Can $\varepsilon$ be taken close to $\diam(C)$?
\item[ii)] In case that the answer is yes, can $C$ be used to get an equivalent renorming of $X$ such that the diameter of the slices of the new unit ball is (close to) $2$?
\end{itemize}
As we have already indicated, the above questions do make sense for the failure of CPCP (resp. SR) if we replace slices with non-empty weakly open sets (resp. convex combinations of slices). Moreover, far from a theoretical interest, the above questions have been analysed in the literature in a series of publications \cite{blr17,ivakhno06,scsewe,ewer}. As a matter of fact, it is proved in \cite[Corollary 3.2]{scsewe} that every Banach space $X$ failing the RNP admits, for every $\varepsilon>0$, an equivalent renorming such that every slice of the new unit ball has diameter, at least, $1-\varepsilon$, leaving the open question whether $2-\varepsilon$ can be raiched (which is a more natural statement) in the above result. Even though an affirmative answer was given in \cite{ewer} for Banach lattices, recent works about the relation between slices with maximal radius and maximal diameter motivates the existence of that gap for the diameter of slices. In order to show this, let us recall a bit of notation: given  Banach space $X$, we say that $X$ has the:
\begin{enumerate}
    \item \textit{slice diameter two property (slice-D2P)} if every slice of $B_X$ has diameter $2$;
    \item \textit{r-big slice property (r-BSP) if every slice $S$ of the unit ball $B_X$ has radius exactly 1, in other words, if given any $x\in X$ and any $\varepsilon>0$ there exists a $y\in S$ such that $\Vert x-y\Vert \geq1-\varepsilon$.}
\end{enumerate}

The slice-D2P has received a lot of attention in during the present century (c.f. e.g. \cite{ahntt16,aln2,blr15eje1,nw01}), whereas the r-BSP was introduced in \cite{ivakhno06} but it remained unnoticed until the recent research we will describe below. From the immediate relation between the radius and the diameter of a set, it is clear that the slice-D2P$\Rightarrow$r-BSP (it was proved in \cite{ivakhno06}). The validity for the converse was posed as an open question in the above mentioned paper by Y. Ivakhno.  A negative answer was provided in \cite[Theorem 3.7]{hllnr20}, where it was proved that the space $X=(JT_\infty)_*$ has the r-BSP but the infimum of the diameter of slices of the unit ball is $\sqrt{2}$. Such counterexample was put further in \cite[Theorem 1.1]{rz23}, where an example of a Banach space $X$ with the r-BSP, the CPCP and such that the infimum of the diameter of slices of the unit ball equals $1$ is constructed. This example shows that the r-BSP and the slice-D2P are as far as they can be because, for a given slice, the radius and the diameter can almost coincide. The above suggests that, if we pursue to characterise the Banach spaces failing the RNP as those admiting an equivalent renoming where all the slices of the unit ball are extremely big, considering the radius is a more natural measure than considering the diameter to define ``big slices''. This is also suggested by \cite[Theorem 2]{ivakhno06}, where it is proved that every Banach space failing the RNP admits, for every $\varepsilon>0$, an equivalent renorming under which every slice of the new unit ball has radius, at least, $1-\varepsilon$. Because of this reason, it is natural to consider the following question.

\begin{question}\label{quest:RNP}
Let $X$ be a Banach space failing the RNP. Is there any equivalent renoming on $X$ satisfying the r-BSP?
\end{question}

Note that, since there are examples of Banach spaces with the CPCP and failing the RNP, an affirmative answer to the above question would imply that Banach spaces with the r-BSP and the CPCP must exists. Such examples, indeed, are known to exist \cite{rz23}. The question whether there exists a Banach space with the slice-D2P and the CPCP is, to the best of our knowledge, an open question \cite[Remark 3.4]{rz23}.

In order to establish an analogue question to Question~\ref{quest:RNP} for the CPCP, we will consider two more properties of Banach spaces related to having big weakly open subsets of the unit ball: given a Banach space $X$, we say that $X$ has the:
\begin{itemize}
    \item \textit{diameter two property (D2P)} if every non-empty relatively weakly open subset of $B_X$ has diameter $2$,
    \item  \textit{r-big weak open property (r-BWOP)} if every non-empty relatively weakly open subset $W$ of the unit ball $B_X$ has radius exactly 1, in other words, if given any $x\in X$ and any $\varepsilon>0$ there exists a $y\in W$ such that $\Vert x-y\Vert \geq 1-\varepsilon$.
\end{itemize}

As well as happen with the slice version, the diameter two property has received a lot of attention in the literature (again we refer to \cite{ahntt16,aln2,blr15eje1,blr15eje2,nw01}). The r-BWOP, however, was recently introduced in \cite{lmr25} in order to show that there are Banach spaces with the r-BWOP and such that the infimum of diameter of slices of the unit ball is $1$ \cite[Theorem 1.1]{lmr25}. All the discussion around Question~\ref{quest:RNP} and the above mentioned \cite[Theorem 1.1]{lmr25} makes the following question to make sense.

\begin{question}\label{ques:CPCP}
Let $X$ be a Banach space failing the CPCP. Is there any equivalent renorming for which every non-empty relatively weakly open subset of $B_X$ has radius $1$?
\end{question}

The similar comment after Question~\ref{quest:RNP} applies here: since there are Banach spaces failing the CPCP and being SR, an affirmative answer to Question~\ref{ques:CPCP} would imply the existence of a SR Banach space with the r-BWOP. 

\begin{question}\label{ques:ejeSRwithrBWOP}
Is there a strongly regular Banach space $X$ with the r-BWOP?
\end{question}

The aim of the present note it to work in the direction of giving a positive answer to the above question. Indeed, our aim is to prove the following result:
\begin{theorem}\label{theo:maintheorem}
Given $1<p<\infty$, there exists a Banach space $Y$ and a closed, convex and symmetric set $L\subseteq B_Y$ with the following properties:
\begin{enumerate}
\item $Y^{**}$ is strongly regular (henceforth, $Y$ is strongly regular). 
\item Every non-empty relatively weakly-star open subset of $\overline{L}^{w^*}$ (the $w^*$ closure of $L$ in $Y^{**}$) has radius one. In particular, every non-empty relatively weakly open subset of $L$ has radius $1$.
\item Every non-empty relatively weakly open subset of $L$ has diameter, at least, $2^\frac{1}{p}$. 
\end{enumerate}
\end{theorem}

As a consequence of the above theorem is the following isomorphic result.

\begin{theorem}\label{theo:renormainejem}
Let $1<p<\infty$ and let $Y$ the Banach space of Theorem~\ref{theo:maintheorem}. Then, given $\varepsilon>0$ there exists an equivalent renorming of $Y$ such that every non-empty relatively $w^*$ open subset of the unit ball of $B_{Y^{**}}$ has radius at least $\frac{1}{1+\varepsilon}$ and diameter at least $\frac{2^\frac{1}{p}}{1+\varepsilon}$. 

In particular, for every $\varepsilon>0$ there exists a SR Banach space $X$ for which every non-empty relatively weakly open subset of $B_X$ has diameter, at least, $2-\varepsilon$.
\end{theorem}

All the sections of the present paper, with the exception of the last one, are devoted to the construction of $Y$ and $L$ in Theorem~\ref{theo:maintheorem}. Since the construction is rather involved, let us describe the content of the sections of the paper by giving a short explanation of the constructions of $Y$ and $L$.

The construction of the space $Y$ in Theorem~\ref{theo:maintheorem} is strongly inspired by the construction of the space $S_*T_\infty$ in \cite[Theorem VI.1]{gms} of a Banach space failing the CPCP and being SR.

The construction in \cite[Section VI]{gms} starts by considering $T_\infty$ the infinite tree $\bigcup\limits_{n\in\mathbb N\cup\{0\}}\mathbb N^n$ and consider $\phi$ to be the origin of the tree. Consider $T_\infty^n:=\{t\in T_\infty: \vert t\vert=n\}$, where $\vert t\vert$ stands for the level in the tree. Then they consider $X:=\left( \oplus_{n=0}^\infty \ell_2(T_\infty^n)\right)_{c_0}$, and find a suitable $w^*$-compact and convex subset $K\subseteq X^{**}$ which satisfies several properties among which they highlight that $K$ is SR but fails the CPCP. 

In Section~\ref{section:weird} we will consider, for $1<p<\infty$, a similar variation of set $K$ constructed in $X:=\left( \oplus_{n=0}^\infty \ell_p(T_\infty^n)\right)_{c_0}$. In such section we will prove that $W:=\frac{K\cap X-K\cap X}{2^\frac{1}{p}}\subseteq B_X$ is strongly regular (here the argument is similar to that of \cite[Section VI]{gms}), that every non-empty relatively $w^*$ open subset of $K$ in $X^{**}$ has radius $1$ (Theorem~\ref{theo:w*openbidual}) and that every non-empty relatively weakly open subset of $K\cap X$ has diameter at least $2^\frac{1}{p}$ (Theorem~\ref{theo:weakopenconjuinter}). Furthermore, we prove that if we consider $Y:=\overline{\spann}(K)\subseteq Y^{**}$, then $Y/X$ is reflexive (here the proof differs from that of \cite[Section VI]{gms}, where the special case $p=2$ permits the authors proving that $Y/X$ is a Hilbert space by finding a suitable scalar product). 
 
Going back to the construction of \cite[Section VI]{gms}, once $X$ and $K$ are constructed there, then $S_*T_\infty$ is originated after applying the interpolation procedure from \cite{dfjp} to $U:=\overline{\conv}(K\cap X\cup-K\cap X)$ and $X$: for every $n\in\mathbb N$ consider $\Vert\cdot\Vert_n$ the (equivalent) norm on $X$ whose unit ball is 
$$U_n:=2^n U+\frac{1}{2^n}B_X,$$
and then 
$$S_*T_\infty=\left\{x\in X: \vert x\vert=\left(\sum_{n=1}^\infty \Vert x\Vert_n \right)^\frac{1}{2}<\infty \right\}.$$

In the proof of Theorem~\ref{theo:maintheorem}, our construction of $Y$ will be a variation of the above argument: once $1<p<\infty$ is fixed and $X$ and $W$ is constructed, for every $n\in\mathbb N$ define $\Vert\cdot\Vert_n$ to be the norm associated to the Minkowski functional of the set
$$B_n:=\overline{\conv}\left(2^n W\cup \frac{1}{2^n}B_X\right).$$
Define $Y:=\left\{x\in X: \sum_{n=1}^\infty \Vert x\Vert_n^2<\infty\right\}$, which is a vector subspace of $X$ and define on $Y$ the following norm
$$\vert\Vert x\Vert\vert:=\left(\sum_{n=1}^\infty \Vert x\Vert_n^2 \right)^\frac{1}{2}.$$
In our construction of $Y$ the modification of the interpolation process is necessary because we do want to keep geometric properties of $W$ under the formal inclusion $i:Y\longrightarrow X$, so we need sharper estimates  for the norm of $i$ and for the radius of the set $i(W)$. This modification, on the other hand, forces us to prove several results for this scheme of interpolation, which are variations of results proved in \cite{dfjp} for the original interpolation process (with similar proofs). All this work is performed in Section~\ref{section:interpola}. 

All the results for $Y$ and $L:=\frac{3}{\sqrt{3}}W\subseteq B_Y$ are proved along Section~\ref{sect:finalresultingsect}. In Theorem~\ref{theo:espafuerteregular} we prove that $Y^{**}$ (and hence $Y$) is strongly regular. Then, we prove in Theorem~\ref{theo:abidebigrandesL} that every non-empty relatively weakly open subset of $L$ has radius $1$ and diameter, at least, $2^\frac{1}{p}$. Indeed, we prove in Theorem~\ref{prop:improbidual} that every non-empty relatively $w^*$ open subset of $\overline{L}^{w^*}$ in $Y^{**}$ has radius one. All this proves Theorem~\ref{theo:maintheorem}. After this, we point out the proof of Theorem~\ref{theo:renormainejem}.

All this constitutes a major progress in the search of an affirmative answer to Question~\ref{ques:ejeSRwithrBWOP} which, as far as we know, is an open question. We conclude the paper with a possible idea for getting a positive solution by involving the use of Tauberian operators and an application of a new interpolation process to the bidual $Y^{**}$ of the resulting space $Y$ from Theorem~\ref{theo:maintheorem} and using some ideas from \cite[Proposition 2.5]{ll21}.

\textbf{Terminology:} We will only consider real Banach spaces. Given a Banach space $X$ we denote by $B_X$ and $S_X$ the unit ball and the unit sphere of $X$ respectively, and we will denote by $X^*$ the topological dual of $X$. 

\section{A strongly regular set with big weak open sets}\label{section:weird}

Our aim is to consider a variation of the construction performed in \cite[p. 578]{gms} of a set which is strongly regular but fails the CPCP. To this end we need a bit of notation.

Throughout the section $1<p<\infty$ will be fixed. 

We consider $T_\infty$ the infinite tree $\bigcup\limits_{n\in\mathbb N}\mathbb N^n$ and consider $\phi$ to be the origin of the tree. Consider $T_\infty^n:=\{t\in T_\infty: \vert t\vert=n\}$, where $\vert t\vert$ stands for the level in the tree. Given $t=(n_1,\ldots, n_k)\in T_\infty$ and $p\in\mathbb N$ we define $t\wedge p:=(n_1,\ldots, n_k, p)$.

Observe that there exists a natural order on $T$: $(n_1,\ldots, n_k)\leq (m_1,\ldots, m_p)$ if $k\leq p$ and $n_i=m_i$ holds for $1\leq i\leq k$. By a \textit{branch} in $T_\infty$ we mean a totally bounded and infinite subset of $T_\infty$.

Consider $X:=\left(\oplus_{n=1}^\infty \ell_p(T_\infty^n)\right)_{c_0}$ and $K\subseteq X^{**}=(\oplus_{n=1}^\infty \ell_p(T_\infty^n))_{\ell_\infty}$ be the set of all the elements $(z_t)_{t\in T_\infty}$ such that 
\begin{enumerate}
    \item $z_\phi=1$
    \item $z_t\geq 0$ holds for every $t\in T_\infty$ and,
    \item $z^p_{(n_1,\ldots, n_k)}\geq \sum_{n_{k+1}=1}^\infty z^p_{(n_1,\ldots, n_k, n_{k+1})}$
\end{enumerate}

Let $Y:=\overline{\spann}(K)\subseteq X^{**}$. Observe that $X\subseteq Y$ since $K\cap X$ is dense in $X$. Indeed, given any $t\in T_\infty$, we have that $e_t=\sum_{s\leq t} e_s-\sum_{s<t} e_s\in Y$ for every $t\in T_\infty$. 

Consequently we can consider the quotient space $Y/X$. 

\begin{proposition}\label{prop:cocienteprincipio}
Given $y+X\in Y/X$ we have that
$$\Vert y+X\Vert=\lim\limits_{n\rightarrow \infty} \left(\sum_{\vert t\vert=n}\vert y(t)\vert^p \right)^\frac{1}{p}.$$
\end{proposition}

Observe that the above limit is well defined on $K$ since, for $x\in K$, the sequence $(\sum_{\vert t\vert=n}\vert x_t\vert^p)$ is decreasing in $n$ by condition (3) in the definition. Consequently the limit is well defined on $\spann (K)$ and, by continuity, on $Y$. 

\begin{proof}
Given $y+X\in Y/X$ then, by definition, $\Vert y+X\Vert=d(y,X)$. 

To begin with, let us prove that $d(y,X)\leq \lim\limits_{n\rightarrow \infty} \left(\sum_{\vert t\vert=n}\vert y(t)\vert^p \right)^\frac{1}{p}=:L$. In order to do so select $\varepsilon>0$. By the definition of the limit there exists $m\in \mathbb N$ such that $n\geq m$ implies
$$\left\vert \left(\sum_{\vert t\vert=n}\vert y(t)\vert^p\right)^\frac{1}{p}-L\right\vert<\varepsilon.$$
Now define $x(t)=y(t)$ if $\vert t\vert\leq m$ and $x(t)=0$ if $\vert t\vert>m$. It is clear that $x\in X$, so $d(y,X)\leq \Vert y-x\Vert$. Moreover, since $x(t)-y(t)=0$ holds if $\vert t\vert\leq m$ we infer that
\[
\begin{split}
\Vert x-y\Vert=\sup_{n\in\mathbb N} \left(\sum_{\vert t\vert=n}\vert x(t)-y(t)\vert^p\right)^\frac{1}{p}& =\sup_{n\geq m+1}\left(\sum_{\vert t\vert=n}\vert y(t)\vert^p\right)^\frac{1}{p}\\
& \leq L+\sup_{n\geq m+1}\left\vert \left(\sum_{\vert t\vert=n}\vert y(t)\vert^p\right)^\frac{1}{p}-L\right\vert\\
& \leq L+\varepsilon.
\end{split}
\]
Consequently we get $\Vert y+X\Vert\leq L+\varepsilon$ and the arbitrariness of $\varepsilon$ implies $\Vert y+X\Vert\leq L$.

To prove the reverse inequality, let $\varepsilon>0$. Since $\Vert y+X\Vert=d(y,X)$ find $x\in X$ such that $\Vert y+X\Vert+\varepsilon\geq \Vert y-x\Vert$. We can assume with no loss of generality that $x$ is finitely supported by a density argument. Now
\[
\begin{split}
L=\lim_{n\rightarrow \infty}\left(\sum_{\vert t\vert=n} \vert y(t)\vert^p\right)^\frac{1}{p}& =\lim_{n\rightarrow \infty}\left(\sum_{\vert t\vert=n} \vert y(t)-x(t)\vert^p\right)^\frac{1}{p}\\
& \leq \sup_{n\in\mathbb N}\left(\sum_{\vert t\vert=n} \vert y(t)-x(t)\vert^p\right)^\frac{1}{p}\\
& =\Vert y-x\Vert\leq \Vert y+X\Vert+\varepsilon,
\end{split}
\]
and once again the arbitrariness of $\varepsilon$ yields $L\leq \Vert y+X\Vert$ and the equality is proved.
\end{proof}

Let us prove that $Y/X$ is a reflexive space. In order to do so, consider
$$\Gamma:=\{\gamma\subseteq T_\infty: \gamma\mbox{ is a branch }\}.$$

\begin{theorem}\label{theo:cocientereflexivo}
The mapping $\phi:Y/X\longrightarrow \ell_p(\Gamma)$ defined by
$$\phi(y+X)(\gamma):=\lim_{t\in\gamma} y(t)$$
is an into isometry. In particular, $Y/X$ is a reflexive Banach space.
\end{theorem}

In order to do so we need the following lemma.

\begin{lemma}\label{lemma:injephi}
Let $\varphi:Y\longrightarrow \ell_p(\Gamma)$ defined by
$$\varphi(y)(\gamma):=\lim_{t\in\gamma} y(t),$$
then $\varphi$ is a well defined, linear and continuous mapping such that
$$\ker(\varphi)=X.$$
\end{lemma}

\begin{proof}
In order to show that $\varphi$ is well defined let us show that it is well defined on $\spann(K)$. Observe that if $y\in K$ then, by definition of $K$, given any $\gamma=\{t_1<t_2<\ldots\}\in\Gamma$, then the sequence $(y(t_n))_{n\in\mathbb N}$ is decreasing and bounded below (by 0), so the above sequence is convergent. Consequently, $\lim_{t\in \Gamma} y(t)=\lim\limits_n y(t_n)$ is well defined for every $y\in \spann(K)$ and every $\gamma\in\Gamma$. 

In order to prove that $\varphi$ is bounded on $\spann(K)$ take any $y\in \spann(K)$ and let any finite set $F=\{\gamma_1,\ldots, \gamma_q\}\subseteq \Gamma$. Since the above are diffe\-rent branches in $T_\infty$ there exists some level $m$ from which the branches are pairwise disjoint, i.e.
$$\gamma_i\cap \gamma_j\cap T_\infty^n=\emptyset \mbox{if } i\neq j, n\geq m.$$
So given any $n\geq m$ we have that
\[
\begin{split}
\left(\sum_{i=1}^q\sum_{t\in \gamma_i\cap T_\infty^n}\left\vert y(t)\right\vert^p  \right)^\frac{1}{p}\leq \left(\sum_{\vert t\vert=n}\vert y(t)\vert^p \right)^\frac{1}{p}&
 \leq \sup_{n\in\mathbb N}\left(\sum_{\vert t\vert=n} \vert y(t)\vert^p \right)^\frac{1}{p}=\Vert y\Vert.
\end{split}
\]
Since the above inequality holds true for $n\geq m$ we get that
\[
\begin{split}
\left( \sum_{i=1}^q \left\vert \lim\limits_{t\in\gamma_i}y(t)\right\vert^p \right)^\frac{1}{p}\leq \Vert y\Vert.
\end{split}
\]
This implies that $\varphi:\spann(K)\longrightarrow \ell_p(\Gamma)$ is a well defined, linear and bounded mapping with $\Vert \varphi\Vert\leq 1$. The completeness of $\ell_p(\Gamma)$ and the density of $\spann(K)$ allows to extend the above mapping to a bounded linear operator $\varphi:Y\longrightarrow \ell_p(\Gamma)$ such that $\Vert \varphi\Vert\leq 1$. 

So it remains to prove that $\ker(\varphi)=X$. The inclussion $\supseteq$ is immediate. In order to prove the reverse inclussion, let $y\in Y$ such that $d(y,X)\geq \varepsilon_0$. A variation of the argument in \cite[Lemma IV.2]{gm85} allows to find a subtree $T_1$ of $T_\infty$ with finitely many branching points which is full (i.e. if $S$ is a segment of $T_\infty$ then $S\cap T_1$ is a segment of $T_1$) and such that $\Vert y-P_{T_1}(y)\Vert<\frac{\varepsilon}{2}$, where $P_{T_1}: Y\longrightarrow \{y\in Y: y(t)=0\ \forall t\in T_\infty\setminus T_1\}$ is the natural projection. 

An adaptation of the technical argument in \cite[p. 12]{bpv19} (where it is performed by a diadyc tree) we conclude that $y\notin \ker(\varphi)$. 
\end{proof}

\begin{proof}[Proof of Theorem~\ref{theo:cocientereflexivo}]
Observe that the mapping $\phi$ is well defined and injective by Lemma~\ref{lemma:injephi}. Moreover $\Vert\phi\Vert\leq 1$. In order to conclude the result let us prove that $\phi$ is a quotient map.

In order to prove that $\phi$ is a quotient map select $z\in \ell_p(\Gamma)\setminus\{0\}$. Select $\supp(z)=\{\gamma_n:n\in\mathbb N\}\subseteq \Gamma$.

Now consider $y_n:=\chi_{\gamma_n}$ the characteristic function on $\gamma_n$. We claim that $y:=\sum_{n=1}^\infty z(\gamma_n)y_n+X$ is well defined in $Y/X$. Indeed, given $k\in\mathbb N$ there exists a level $q$ large enough to get that $\gamma_i\cap \gamma_j\cap \{t: \vert t\vert=q\}=\emptyset$ holds for $1\leq i,j\leq k$. It follows, by the expression of the norm on $Y/X$ described in Proposition~\ref{prop:cocienteprincipio} that
$$\left\Vert \sum_{i=1}^k z_i(\gamma_i) y_i+X \right\Vert=\left(\sum_{i=1}^k \vert z(\gamma_i)\vert^p\right)^\frac{1}{p}.$$
By the above argument it is clear that the sequence $\left\{\sum_{i=1}^k z(\gamma_i)y_i \right\}_{k\in\mathbb N}$ is Cauchy in $Y/X$ and, indeed,
$$\left\Vert \sum_{i=1}^k z(\gamma_i)y_i-\sum_{i=1}^{q}z(\gamma_i)y_i+X \right\Vert=\left(\sum_{i=k+1}^q \vert z(\gamma_i)\vert^p\right)^\frac{1}{p},$$
which clearly implies Cauchy condition. The completeness of $Y/X$ implies that there exists $y:=\sum_{n=1}^\infty z(\gamma_n)y_n\in Y/X$. Moreover it is immediate that $\phi(y)=z$ and that $\Vert y\Vert\leq \Vert z\Vert$, from where $\phi$ is an isometric quotient map.

\end{proof}

With the above result in mind we obtain the following result.

\begin{proposition}
$K$ is strongly regular.
\end{proposition}

\begin{proof}
Observe that $X$ does not contain $\ell_1$ isomorphically since $X^*$ is separable. On the other hand, $Y/X$ does not contain $\ell_1$ since $Y/X$ is reflexive by the above result.

Consequently, $Y$ does not contain $\ell_1$ since the property of containig $\ell_1$ is a three-space property (c.f. e.g. \cite[Theorem 3.2 d.]{cm}). Now \cite[Proposition VI.3]{gms} forces that $K$ is strongly regular. 
\end{proof}

Let $K_0:=K\cap X$. Let $\mathcal U$ be a non-empty weakly open subset of $K_0$. We can assume up to taking a smaller weakly open set that
$$\mathcal U=V_{A,\varepsilon}:=\left\{\xi\in K_0: \vert x(t)-\xi(t)\vert<\varepsilon\ \forall t\in A \right\},$$
where $A\subseteq T_\infty$ is finite and full (i.e. $t\in A$ and $s\leq t\Rightarrow s\in A$). Observe that we can make this assumption since $X^*=\left(\oplus_{n=1}^\infty \ell_{p^*}(T_\infty^n) \right)_1$.

A similar argument to the proof of \cite[Proposition VI.2.a)]{gms} allows to prove that, given $n\in\mathbb N$ such that $\vert t\vert<n$ holds for every $t\in A$, there exists $\xi\in \mathcal U$ such that 
$$\sum_{\vert t\vert=n} \vert \xi(t)\vert^p=1.$$

Now define the following set 
\begin{equation}\label{conjuntoW}W:=\frac{K_0-K_0}{2^\frac{1}{p}}.
\end{equation} $W$ is a symmetric closed convex subset of $B_X$. Now we have the following result.

\begin{theorem}\label{theo:weakopenconjuinter}
Let $\mathcal U$ be a non-empty relatively weakly open subset of $W$. Then:
\begin{enumerate}
    \item $\mathcal U$ has radius $1$.
    \item $\diam(\mathcal U)\geq 2^\frac{1}{p}.$
\end{enumerate}
\end{theorem}

\begin{proof}
By the continuity of the sum we can find two weak open subsets $V_1$ and $V_2$ of $K_0$ such that
$$\frac{V_1-V_2}{2^p}\subseteq \mathcal U.$$
Moreover we can assume up to taking smaller sets that
$$V_i=V_{A,\varepsilon,i}:=\left\{\xi\in K_0: \vert x(t)-\xi_i(t)\vert<\varepsilon\ \forall t\in A \right\}, i=1,2,$$
where $A\subseteq T_\infty$ is finite and full.

Let us begin by proving that $\frac{V_1-V_2}{2^\frac{1}{p}}$ has radius $1$. In order to do so, pick any $x\in X$ and we can assume, up to a density argument, that $x$ has finite support. Now we can take $n$ large enough to guarantee that $\max(\{\vert t\vert: t\in A\})<n$ and $\max(\{\vert t\vert: t\in \supp(x)\})<n$ and find $\xi_i\in V_i$ satisfying that $\sum_{\vert t\vert=n}\vert \xi_i(t)\vert^p=1$. Call $B_i:=\{t\in T_\infty^n: \xi_i(t)\neq 0\}$ and define $\phi_1\in V_1, \phi_2\in V_2$ by the equations
$$\phi_i(t):=\left\{\begin{array}{cc}
  \xi_i(t)   & \vert t\vert\leq n \\
  \phi_i(t\wedge i)=\xi(t)   & t\in B_i\\
  0 & \mbox{otherwise}.
\end{array} \right.\hspace{1cm} i=1,2.$$
It follows that $\frac{\phi_1-\phi_2}{2^\frac{1}{p}}\in \frac{V_1-V_2}{2^\frac{1}{p}}\subseteq \mathcal U$. Moreover, since $\supp(x)\cap \{t\in T_\infty: \vert t\vert=n+1\}=\emptyset$ we get
\[
\begin{split}
\left\Vert \frac{\phi_1-\phi_2}{2^\frac{1}{p}}-x \right\Vert^p& \geq \sum_{\vert t\vert=n+1}\frac{\vert \phi_1(t)-\phi_2(t)\vert^p}{2}=\sum_{i=1}^2\sum_{t\in B_i}\frac{\vert \phi_1(t\wedge i)-\phi_2(t\wedge i)\vert^p}{2}\\
& =\frac{1}{2}\sum_{i=1}^2\sum_{t\in B_i}\vert\phi_i(t\wedge i)\vert^p=\frac{1}{2}\sum_{i=1}^2\sum_{t\in B_i}\vert \xi_i(t)\vert^p\\ & =\frac{1}{2}\sum_{i=1}^2\sum_{\vert t\vert=n}\vert \xi_i(t)\vert^p =1,\end{split}
\]
as desired.

In order to prove that the diameter of $\frac{V_1-V_2}{2^\frac{1}{p}}$ is greater than $2^\frac{1}{p}$ we follow the same argument as above defining:
$$\phi_i(t):=\left\{\begin{array}{cc}
  \xi_1(t)   & \vert t\vert\leq n \\
  \phi_i(t\wedge i)=\xi_1(t)   & t\in B\\
  0 & \mbox{otherwise},
\end{array} \right.\hspace{1cm} i=1,3.$$
and
$$\phi_i(t):=\left\{\begin{array}{cc}
  \xi_2(t)   & \vert t\vert\leq n \\
  \phi_i(t\wedge i)=\xi_2(t)   & t\in B\\
  0 & \mbox{otherwise}.
\end{array} \right.\hspace{1cm} i=2,4.$$

Now $\frac{\phi_1-\phi_2}{2^\frac{1}{p}}$, $\frac{\phi_3-\phi_4}{2^\frac{1}{p}}\in \frac{V_1-V_2}{2^\frac{1}{p}}\subseteq\mathcal U$. 

Since the supports of $\phi_i$ and $\phi_j$, when intersected with $T^{n+1}_\infty$, are pairwise disjoint, we get that
\[
\begin{split}
& \left\Vert \frac{\phi_1-\phi_2}{2^\frac{1}{p}}-\frac{\phi_3-\phi_4}{2^\frac{1}{p}} \right\Vert^p\geq\\
& \frac{\sum_{t\in B}\vert \phi_1(t\wedge 1)\vert^p+\sum_{t\in B}\vert \phi_2(t\wedge 2)\vert^p+\sum_{t\in B}\vert \phi_3(t\wedge 3)\vert^p+\sum_{t\in B}\vert \phi_4(t\wedge 4)\vert^p}{2}\\
& =2,  
\end{split}
\]
and the result follows.
\end{proof}

Let us end getting an improvement of (1) in Theorem~\ref{theo:weakopenconjuinter} by proving that every non-empty relatively $w^*$ open subset of $\overline{W}^{w^*}$ has radius $1$.

\begin{theorem}\label{theo:w*openbidual}
Let $\mathcal U$ be a non-empty relatively $w^*$ open subset of $\overline{W}^{w^*}\subseteq B_{X^{**}}$. Then $r(\mathcal U)=1$.
\end{theorem}

\begin{proof}
By the continuity of the sum we can find a weak-star open set of $K$ such that
$$\frac{V_1-V_2}{2^p}\subseteq \mathcal U.$$
Moreover we can assume up to taking a smaller set $V$ that
$$V_i=V_{A,\varepsilon,i}:=\left\{\xi\in K: \vert x(t)-\xi_i(t)\vert<\varepsilon\ \forall t\in A \right\},$$
where $A\subseteq T_\infty$ is finite and full.

In order to prove that $\frac{V_1-V_2}{2^\frac{1}{p}}$ has radius $1$, pick any $x\in X^{**}$ and $\varepsilon>0$. Now we can find $n$ large enough to guarantee that $\max(\{\vert t\vert: t\in A\})<n$ and $\max(\{\vert t\vert: t\in \supp(x)\})<n$ and find $\xi_i\in V_i$ satisfying that $\sum_{\vert t\vert=n}\vert \xi_i(t)\vert^p=1$. Call $B_i:=\{t\in T_\infty^n: \xi_i(t)\neq 0\}$. Now, since $X^{**}=(\oplus_{n=1}^\infty \ell_p(T_\infty^n))_{\ell_\infty}$ we get that, given $t\in B_1\cup B_2$, then
$$\sum_{n=1}^\infty \vert x^{**}(t\wedge n)\vert^p\leq \Vert x^{**}\Vert,$$
which clearly implies that $x^{**}(t\wedge n)\rightarrow 0\ (n\rightarrow \infty)$. Consequently we can find, for every $t\in B_1\cup B_2$ a natural $k_t\in\mathbb N$ such that $n\geq k_t$ implies 
\begin{equation}\label{eq:condivalsuc}
\left(1-2^\frac{1}{p}\frac{\vert x^{**}(t\wedge n)\vert}{\max_{i=1,2} \vert \xi_i(t)\vert}\right)^p>1-\varepsilon.
\end{equation}

Now define $\phi_1\in V_1,\phi_2\in V_2$ by the equations
$$\phi_i(t):=\left\{\begin{array}{cc}
  \xi_i(t)   & \vert t\vert\leq n \\
  \phi_i(t\wedge k_t+i)=(-1)^{i+1}\xi_i(t)   & t\in B\\
  0 & \mbox{otherwise}.
\end{array} \right.\hspace{1cm} i=1,2.$$
As in the proof of Theorem~\ref{theo:weakopenconjuinter} it follows that $\frac{\phi_1-\phi_2}{2^\frac{1}{p}}\in \frac{V_1-V_2}{2^\frac{1}{p}}\subseteq \mathcal U$. Moreover, also following the final estimates of such result we infer that
\[
\begin{split}
\left\Vert \frac{\phi_1-\phi_2}{2^\frac{1}{p}}-x \right\Vert^p& \geq \sum_{\vert t\vert=n+1}\left\vert \frac{\phi_1(t)-\phi_2(t)}{2^\frac{1}{p}}-x(t)\right\vert^p
\\ & \geq \sum_{i=1}^2\sum_{t\in B}\left\vert \frac{\phi_1(t\wedge k_t+ i)-\phi_2(t\wedge k_t+i)}{2^\frac{1}{p}}-x(t\wedge k_t+i)\right\vert^p\\
& =\sum_{i=1}^2 \sum_{t\in B_i}  \left\vert \frac{\xi_i(t)}{2^\frac{1}{p}}-x(t\wedge k_t+i)\right\vert^p \geq \sum_{i=1}^2\sum_{t\in B} \left(\left\vert \frac{\xi_i(t)}{2^\frac{1}{p}}\right\vert -\vert x(t\wedge k_t+i)\vert \right)^p \\
& = \sum_{i=1}^2\sum_{t\in B_i} \frac{\vert\xi_i(t)\vert^p}{2}\left(1-2^\frac{1}{p}\frac{\vert x(t\wedge k_t+i)\vert}{\vert\xi_i(t)\vert }\right)^p
\end{split}
\]
By \eqref{eq:condivalsuc} we get from the above that
\[\begin{split}
\sum_{i=1}^2\sum_{t\in B_i} \frac{\vert\xi_i(t)\vert^p}{2}\left(1-2^\frac{1}{p}\frac{\vert x(t\wedge k_t+i)\vert}{\vert\xi_i(t)\vert } \right)^p & >(1-\varepsilon)\sum_{i=1}^2\sum_{t\in B_i} \frac{\vert \xi_i(t)\vert^p}{2}\\
& =\frac{(1-\varepsilon)}{2}\sum_{i=1}^2\sum_{t\in B_i} \vert \xi_i(t)\vert^p=1-\varepsilon
\end{split}\]
To sum up we have proved that given any $\varepsilon>0$ and any $x^{**}\in X^{**}$ there exists $u\in\mathcal U$ such that $\Vert x^{**}-u\Vert>1-\varepsilon$. This proves that $r(\mathcal U)=1$, as desired.
\end{proof}

Let us conclude with an observation concerning the relation between the radius of a set and the one of its $w^*$ closure in the bidual space. Given a Banach space $X$, if we consider a set $C\subseteq X$ then, by a density argument and because of the $w^*$ lower semicontinuity of a dual norm, it is not difficult to prove that
$$\diam(C)=\diam(\overline{C}^{w^*}).$$
However, the radius of $C$, when viewed as a subset of $X$, may differ from the radius of $\overline{C}^{w^*}$ when viewed in $X^{**}$, as the following example shows.

\begin{example}
Let $C:=\overline{\conv}(\{e_n:n\in\mathbb N\}
)\subset c_0$. It is not difficult to prove that $r(C)=1$. However, $C\subseteq B((1/2,1/2,\ldots), 1/2)$ since every element in $C$ is positive, so $\overline{C}^{w^*}\subseteq B((1/2,1/2,\ldots), 1/2)$. 
\end{example}

Indeed, at this point it is a natural question whether there exists a Banach space $X$ with the rBWOP such that there are non-empty relatively $w^*$ open subsets of $B_{X^{**}}$ of radius strictly smaller than $1$. The answer is affirmative, as the following result shows.

\begin{remark}\label{rem:rBWOPnotequibidual}
In \cite[Lemma 3.1]{lmr25} it is proved that there exists a closed, convex subset $K_0\subseteq S_{c_0^+}$ such that, for every $\varepsilon>0$, if we consider $K:=K_0\times\{1\}\subseteq c_0\oplus_\infty\mathbb R$, then the norm $\vert\cdot\vert_\varepsilon$ on $c_0\oplus\mathbb R$ whose unit ball is given by
$$B:=\overline{\co}(K\cup -K\cup ((1-\varepsilon)B_{c_0\oplus_\infty \mathbb R}+\varepsilon B_{c_0\times\{0\}})),$$
satisfies that $\Vert\cdot\Vert_\infty \leq \vert\cdot\vert_\varepsilon\leq \frac{1}{1-\varepsilon}\Vert\cdot\Vert_\infty$ and that $X:=(c_0\oplus\mathbb R,\vert\cdot\vert_\varepsilon)$ has the rBWOP. We claim that $B_{X^{**}}$ has non-empty relatively $w^*$ open subsets of radius smaller than $1$. Let us observe first that
$$B_{X^{**}}= \co(\overline{K}^{w^*}\cup -\overline{K}^{w^*}\cup ((1-\varepsilon)B_{\ell_\infty\oplus_\infty\mathbb R}+\varepsilon B_{\ell_\infty\times \{0\}})).$$
The inclusion $\subseteq$ is clear since the set of the right side is $w^*$ compact and contains $B_X$ clearly, so it contains $\overline{B_X}^{w^*}=B_{X^{**}}$. The reverse inclusion follows since the set $\co(K\cup -K\cup ((1-\varepsilon)B_{c_0\oplus_\infty \mathbb R}+\varepsilon B_{c_0\times\{0\}}))$, which is clearly contained in $B_{X^{**}}$, is dense in the right side set.

Let $\delta>0$ and consider the following $w^*$ open set
$$S:=\{(x,\alpha)\in B_{X^{**}}: (0,1)(x,\alpha)=\alpha >1-\delta\}.$$
It is clear that $S$ is a non-empty relatively weakly open subset of $B_{X^{**}}$ (indeed it is a $w^*$ slice). Given $(x,\alpha)\in S$ then, since $(x,\alpha)\in B_{X^{**}}=\co(\overline{K}^{w^*}\cup -\overline{K}^{w^*}\cup ((1-\varepsilon)B_{\ell_\infty\oplus_\infty\mathbb R}+\varepsilon B_{\ell_\infty\times \{0\}}))$, we can find $\lambda_1, \lambda_2, \lambda_3 \geq 0$ with $\sum \lambda_i =  1$, $a, b \in \overline{K_0}^{w^*}$, $(x_0, \alpha_0)\in B_{\ell_\infty\oplus_\infty\mathbb R}$ and $x_1\in B_{\ell_\infty}$ satisfying
\begin{align*}
& (x, \alpha) = \lambda_1(a,1) - \lambda_2(b,1) + \lambda_3 [(1-\eps)(x_0, \alpha_0) + \eps(x_1, 0)].
\end{align*}
Since $(x,\alpha) \in S$ we get
$$1-\delta < \lambda_1 - \lambda_2 + \lambda_3(1-\eps)\alpha_0 \leq \lambda_1 - \lambda_2 + (1-\eps) \lambda_3 = 1 - 2 \lambda_2 - \eps\lambda_3,$$
so $2 \lambda_2 + \eps\lambda_3 < \delta$ and thus
\begin{align}\label{17}
\lambda_2 + \lambda_3 = \dfrac{1}{\eps}(\eps \lambda_2 + \eps \lambda_3) < \dfrac{1}{\eps}(2 \lambda_2 + \eps \lambda_3) < \dfrac{\delta}{\eps}
\end{align}
or, equivalently,
\begin{align}\label{18}
\lambda_1 > 1- \dfrac{\delta}{\eps}.
\end{align}
Define $v_0:=((1/2,1/2,1/2,\ldots),1)\in \ell_\infty\oplus\mathbb R=X^{**}$. Now, since $a\in\overline{K_0}^{w^*}$ then it is a positive element of $B_{\ell_\infty}$. Hence $\Vert a-(1/2,1/2,\ldots)\Vert_{\ell_\infty}\leq \frac{1}{2}$. Thus
$$\Vert (a,1)-v_0\Vert_\infty\leq \frac{1}{2}.$$
On the other hand,
$$\Vert (b,1)+v_0\Vert_\infty\leq 2$$
by the triangle inequality. Similarly
$$\Vert (1-\eps)(x_0, \alpha_0) + \eps(x_1, 0) -v_0\Vert_\infty\leq 2.$$
Thus
\[\begin{split}
\Vert (x,\alpha)-v_0\Vert_\infty& =\left\Vert  \lambda_1(a,1) - \lambda_2(b,1) + \lambda_3 [(1-\eps)(x_0, \alpha_0) + \eps(x_1, 0)]-v_0 \right\Vert_\infty\\
& = \Vert  \lambda_1((a,1)-v_0) - \lambda_2 ((b,1)+v_0) \\
& + \lambda_3 [(1-\eps)(x_0, \alpha_0) + \eps(x_1, 0)-v_0] \Vert_\infty\\
& \leq \lambda_1 \Vert (a,1)-v_0\Vert_\infty+\lambda_2 \Vert (b,1)+v_0\Vert_\infty\\
& +\lambda_3 \Vert (1-\eps)(x_0, \alpha_0) + \eps(x_1, 0) -v_0\Vert_\infty\\
& <\lambda_1 \frac{1}{2}+2(\lambda_2+\lambda_3)\leq \frac{1}{2}+2\frac{\delta}{\varepsilon}.
\end{split}\]
From the equivalence constant we get that
$$\vert (x,\alpha)-v_0\vert_\varepsilon\leq \frac{\Vert (x,\alpha)-v_0\Vert_\infty}{1-\varepsilon}<\frac{\frac{1}{2}+2\frac{\delta}{\varepsilon}}{1-\varepsilon}.$$
This implies that $r(S)\leq \frac{\frac{1}{2}+2\frac{\delta}{\varepsilon}}{1-\varepsilon}$, which concludes the result since $\varepsilon$ and $\delta$ can be taken as close to $0$ as desired.
\end{remark}

Observe that the above establishes a big difference between the rBWOP and the D2P as it is well known that a Banach space $X$ has the D2P if, and only if, every non-empty $w^*$ open subset of $B_{X^{**}}$ has diameter 2.

\section{A modificacion on the interpolation process}\label{section:interpola}

As we described in the introduction, for the proof of Theorem~\ref{theo:maintheorem} we need to consider a modification of the interpolation process described in \cite{dfjp} in order to obtain a similar procedure which allows us to making certain geometric arguments. The aim of this section is to present such interpolation process and to prove many results which we will need in the next section. In the proofs of the new results we will get, however, inspiration by the results of \cite{dfjp}.

Let $X$ be a Banach space and let $W\subseteq B_X$ be a symmetric, closed and convex subset. For every $n\in\mathbb N$ define $\Vert\cdot\Vert_n$ to be the norm associated to the Minkowski functional of the set
$$B_n:=\overline{\conv}\left(2^n W\cup \frac{1}{2^n}B_X\right).$$
Define $Y:=\left\{x\in X: \sum_{n=1}^\infty \Vert x\Vert_n^2<\infty\right\}$, which is a vector subspace of $X$ and define on $Y$ the following norm
$$\vert\Vert x\Vert\vert:=\left(\sum_{n=1}^\infty \Vert x\Vert_n^2 \right)^\frac{1}{2}.$$
Finally denote by $B_Y:=\{x\in Y: \vert \Vert x\Vert\vert\leq 1\}$. 

Our approach differs from that of \cite{dfjp} in the sense that they consider in the above construction the set $B_n=2^n W+\frac{1}{2^n}B_X$. However, we will be able to make a big transfer of information from the results of \cite{dfjp} to our setting.

Let us start with the some elementary results, whose proofs are similar to that of \cite[Lemma 1, (i),(ii),(iii) and (iv)]{dfjp}.

\begin{proposition}\label{prop:intercuentaini}
\begin{enumerate}
    \item $W\subseteq \frac{\sqrt{3}}{3}B_Y$.
    \item The formal identity mapping $j:Y\longrightarrow X$ is continuous and $\Vert j\Vert\leq \frac{3}{\sqrt{3}}=\sqrt{3}$.
    \item $(Y,\vert\Vert \cdot\Vert \vert)$ is a Banach space.
    \item $j^{**}:Y^{**}\longrightarrow X^{**}$ is injective and $(j^{**})^{-1}(X)=Y$.

    \item If $W$ is relatively weakly compact then $Y$ is reflexive.
\end{enumerate}
\end{proposition}

\begin{proof}
(1) Observe that given $w\in W$ we get that $\Vert w\Vert_n\leq \frac{1}{2^n}$. Consequently, we infer that
$$\vert \Vert w\Vert\vert^2=\sum_{n=1}^\infty \Vert w\Vert_n^2\leq \sum_{n=1}^\infty \left(\frac{1}{2^n}\right)^2=\sum_{n=1}^\infty \frac{1}{4^n}=\frac{\frac{1}{4} }{1-\frac{1}{4}}=\frac{\frac{1}{4}}{\frac{3}{4}}=\frac{1}{3}.$$
Consequently $\vert \Vert w\Vert\vert\leq \frac{\sqrt{3}}{3}$, as desired.

(2) Observe that given $n\in\mathbb N$ it is clear that
$$B_n=\overline{\conv}\left(2^n W\cup \frac{1}{2^n}B_X\right)\subseteq 2^n B_X$$
since $W\subseteq B_X$. Taking the above into account, given $y\in Y\setminus\{0\}$ we get
$$\frac{y}{\Vert y\Vert_n}\in B_n\subseteq 2^n B_X.$$
From the above we get
$$\left\Vert \frac{y}{\Vert y\Vert_n}\right\Vert=\frac{\Vert y\Vert}{\Vert y\Vert_n}\leq 2^n\Rightarrow \Vert y\Vert_n\geq \frac{1}{2^n}\Vert y\Vert.$$
Now
$$\vert \Vert y\Vert\vert^2=\sum_{n=1}^\infty \Vert y\Vert_n^2\geq \sum_{n=1}^\infty \left(\frac{1}{2^n}\Vert y\Vert \right)^2=\Vert y\Vert^2\sum_{n=1}^\infty \frac{1}{4^n}=\Vert y\Vert^2\frac{1}{3}.$$
From the above we get that
$$\Vert j(y)\Vert=\Vert y\Vert\leq \sqrt{3}\vert \Vert y\Vert\vert, $$
proving that $j$ is continuous and that $\Vert j\Vert\leq \sqrt{3}$. 

(3) Let $\varphi: (Y,\vert\Vert \cdot\Vert\vert)\longrightarrow (\oplus_{n=1}^\infty (X,\Vert\cdot\Vert_n))_2$ by 
$$\varphi(y):=(y,y,y,\ldots )\ \ y\in Y.$$
It is immediate from the very definition of the norm on $Y$ that $\varphi$ is an into isometry onto the subspace $\{(x_n)\in (\oplus_{n=1}^\infty (X,\Vert\cdot\Vert_n)_2: x_n=x_1\ \forall n\in\mathbb N\}$, which is clearly closed (and hence complete). Consequently $Y$ is complete.

(4) Observe that $\varphi^{**}:Y^{**}\longrightarrow (\oplus_{n=1}^\infty (X,\Vert\cdot\Vert_n)^{**})_2$ is such that $\varphi^{**}(y^{**})=(j^{**}(y^{**}),j^{**}(y^{**}),\ldots)$. Since $\varphi$ is an isometry then $\varphi^{**}$ is injective, so $(\varphi^{**})^{-1}(X)=(\varphi^{**})^{-1}(\varphi(Y))=Y$.

(5) On the one hand observe that the closure of $j(B_Y)$ in $\sigma(X^{**},X^*)$ in $X^{**}$ is $j^{**}(B_{Y^{**}})$. In fact, observe that $B_{Y^{**}}$ is $\sigma(Y^{**},Y^*)$ compact, so $j^{**}(B_{Y^{**}})$ is $\sigma(X^{**},X^*)$ compact as $j^{**}$ is $w^*-w^*$ continuous. In particular $j^{**}(B_{Y^{**}})$ is $\sigma(X^{**},X^*)$ closed.

On the other hand, since $B_Y$ is $\sigma(Y^{**},Y^*)$ dense in $B_{Y^{**}}$ we get that $j^{**}(B_Y)=B_Y$ is $\sigma(X^{**},X^*)$ dense in $j^{**}(B_{Y^{**}})$.

Now assume that $W$ is relatively weakly compact, that is, $\overline{W}$ is weakly compact. Observe that, given $n\in\mathbb N$, the set
$$\conv\left(2^n \overline{W}\cup \frac{1}{2^n}B_{X^{**}} \right)$$
contains $B_Y$ and it is weak-star closed (since both $W$ and $B_{X^{**}}$ are weak-star compact). Consequently, $j^{**}(B_{Y^{**}})\subseteq \overline{\conv}\left(2^n W\cup \frac{1}{2^n}B_{X^{**}} \right)$. Hence
\[
\begin{split}j^{**}(B_{Y^{**}})\subseteq \bigcap\limits_{n\in\mathbb N} \conv\left(2^n \overline{W}\cup \frac{1}{2^n}B_{X^{**}} \right)& \subseteq \bigcap\limits_{n\in\mathbb N} 2^n \overline{W}+\frac{1}{2^n}B_{X^{**}}\\
& \subseteq \bigcap\limits_{n\in\mathbb N} X+\frac{1}{2^n}B_{X^{**}}=X.
\end{split}\]
Thus $j^{**}(B_{Y^{**}})\subseteq X$. Hence
$$B_{Y^{**}}\subseteq (j^{**})^{-1}(j^{**}(B_{Y^{**}}))\subseteq (j^{**})^{-1}(X)=Y.$$
So $B_{Y^{**}}\subseteq Y$ from where $Y^{**}=Y$ and $Y$ is reflexive.
\end{proof}

Let us continue with more results which are transferred from \cite{dfjp}. Before beginning the proof we will denote $X_n:=(X,\Vert \cdot\Vert_n)$ and, given any Banach space $V$, we denote by $b_V$ the open unit ball of $V$.

\begin{proposition}\label{prop:interpolcociente}

\begin{enumerate}
\item $B_{X_n^{**}}=\conv\left(2^n \overline{W}^{w^*}\cup \frac{1}{2^n}B_{X^{**}} \right)$.

\item The open unit ball of $X_n^{**}$, denoted by $b_{X_n^{**}}$, equals $\conv\left(2^n \overline{W}^{w^*}\cup \frac{1}{2^n}b_{X^{**}}\right)$
\item  The space $Y^{**}/Y$ is isometrically a subspace of $Y_0$ obtained by applying the construction to $X_0:=X^{**}/X$ and $W_0:=Q(\overline{W}^{w^*})$, where $Q:X^{**}\longrightarrow X^{**}/X$ stands for the natural quotient operator.
\end{enumerate}
\end{proposition}

\begin{proof}

The proof of (1) is quite straigthforward. Indeed, from the equality $\conv\left(2^n \overline{W}^{w^*}\cup \frac{1}{2^n}B_{X^{**}} \right)=\{\lambda u+(1-\lambda)v:\lambda\in [0,1], u\in 2^n\overline{W}^{w^*}, v\in \frac{1}{2^n}B_{X^{**}}\}$ it is clear that the above set is a weak-star compact subset of $X^{**}$ which contains $B_{X_n}=\overline{\conv}\left(2^n W\cup \frac{1}{2^n}B_{X} \right)$, and since $\overline{B_{X_n}}^{w^*}=B_{X_n^{**}}$ the inclusion $B_{X_n^{**}}\subseteq \conv\left(2^n \overline{W}^{w^*}\cup \frac{1}{2^n}B_{X^{**}} \right)$ follows.

The reverse inequality follows since $\conv\left(2^n W\cup \frac{1}{2^n}B_{X} \right)$, which is weak-star dense in $\conv\left(2^n \overline{W}^{w^*}\cup \frac{1}{2^n}B_{X^{**}} \right)$, is clearly contained in $B_{X_n}$. Since $B_{X_n^{**}}$ is weak-star closed the inclussion $\conv\left(2^n \overline{W}^{w^*}\cup \frac{1}{2^n}B_{X^{**}} \right)\subseteq B_{X_n^{**}}$ is clear and the proof of (1) is finished.

For the proof of (2), call $V=\conv\left(2^n \overline{W}^{w^*}\cup \frac{1}{2^n}b_{X^{**}}\right)$. Let us begin by proving that $V$ is open. Indeed, given $z^{**}\in X^{**}$, it follows that $z^{**}\in V$ if, and only if, there exist $v^{**}\in 2^n\overline{W}^{w^*}, u^{**}\in \frac{1}{2^n}b_{X^{**}}$ and $\lambda\in [0,1]$ such that $z^{**}=\lambda v^{**}+(1-\lambda)u^{**}$. Consequently
$$V=\bigcup\limits_{\lambda\in [0,1]}\bigcup\limits_{v^{**}\in \overline{W}^{w^*}} \lambda v^{**}+(1-\lambda)\frac{1}{2^n}b_{X^{**}},$$
and the above is an open set since it is the union of open sets. 

Taking into account the above and (1) we get that
$$B_{X_n^{**}}=\overline{\conv\left(2^n \overline{W}^{w^*}\cup \frac{1}{2^n}b_{X^{**}} \right)}.$$
Taking interiors we get
\[\begin{split}b_{X_n^{**}}& =\innt(B_{X_n^{**}})=\innt\left(\overline{\conv\left(2^n \overline{W}^{w^*}\cup \frac{1}{2^n}b_{X^{**}} \right)} \right)\\
& =\innt\left(\conv\left(2^n \overline{W}^{w^*}\cup \frac{1}{2^n}b_{X^{**}} \right)\right)\\
& =\conv\left(2^n \overline{W}^{w^*}\cup \frac{1}{2^n}b_{X^{**}} \right)=V,
\end{split}\]
where we have used that the interior and the interior of the closure of a convex set agree.

Finally, in order to prove (3), let us begin with the identification of $X_n^{**}/X_n$ and $(X_0)_n=(X^{**}/X)_n$. If we consider the open unit balls, we infer by (2) that
$$b_{X_n^{**}/X_n}=Q(b_{X_n^{**}})=Q\left(\conv\left(2^n \overline{W}^{w^*}\cup \frac{1}{2^n}b_{X^{**}} \right)\right),$$
whereas
$$b_{(X_0)_n}=\conv\left(2^n Q(\overline{W}^{w^*})\cup \frac{1}{2^n} b_{(X_0)_n^{**}}\right).$$
Now the identification is clear since $Q(b_{X_n^{**}})=b_{(X_0)_n}$

Now set $Z:=\left(\oplus_{n=1}^\infty X_n\right)_2$. Then $Z^{**}=\left(\oplus_{n=1}^\infty X_n^{**}\right)_2$.

Consequently, if we consider the isometric embedding $\varphi: Y\longrightarrow Z$ defined by $\varphi(y):=(y,y,y,\ldots)$, it follows that $\varphi^{**}:Y^{**}\longrightarrow Z^{**}$ is an isometric embedding too. Consider $P:Z^{**}\longrightarrow Z^{**}/Z$ the natural quotient map. Observe that clearly $Y= \ker(P\circ \varphi^{**})$ (since $\varphi^{**}(y)=\varphi(y)\in Z$). Now the above induces a mapping $\bar\varphi:Y^{**}/Y\longrightarrow Z^{**}/Z$ which is isometric.

Now consider $q:Z^{**}\longrightarrow \left(\oplus_{n=1}^\infty X_n^{**}/X_n \right)_2$ by the equation
$$q(x_n^{**})=(x_n^{**}+X_n).$$
Observe that $(x_n)\in \ker(q)$ if, and only if, $x_n\in X_n$ holds for every $n\in\mathbb N$, which means that $(x_n)\in Z$. 

Thus $q$ defines an isometry $\bar q:Z^{**}/Z\longrightarrow Z_0=\left(\oplus_{n=1}^\infty X_n^{**}/X_n\right)_2$.

Now $\bar q\circ\bar \varphi: Y^{**}/Y\longrightarrow Z_0$ is an isometric embedding in the diagonal of $Z_0$, which proves the result.
\end{proof}

\section{Proof of Theorem~\ref{theo:maintheorem}}\label{sect:finalresultingsect}
Set $1<p<\infty$.

Let $W$ be the set defined in \eqref{conjuntoW} and consider $Y$ to be the Banach space resulting of applying the interpolation process associated to $W=\frac{K_0-K_0}{2^\frac{1}{p}}$ and $X:=\left(\oplus_{n=1}^\infty \ell_p(T_\infty^n)\right)_{c_0}$ described in Section~\ref{section:interpola}.

\begin{theorem}\label{theo:espafuerteregular}
$Y^{**}$ is strongly regular.
\end{theorem}

For the proof we will need a pair of lemmata.

\begin{lemma}\label{lemma:separadual}
$Y^*$ is separable.    
\end{lemma}

\begin{proof}
Set $Z:=\left(\oplus_{n=1}^\infty X_n\right)_2$, and consider tha mapping $\varphi:Y\longrightarrow Z$ by the equation
$$\varphi(y):=(y,y,y,\ldots).$$
By the proof of Proposition~\ref{prop:intercuentaini} (3) we get that $\varphi$ is an into isometry, so $\varphi^*:Z^*=\left(\oplus_{n=1}^\infty X_n^* \right)_2\longrightarrow Y^*$ is a quotient map. Since $Z^*$ is separable as $X_n^*$ is separable for every $n\in\mathbb N$ ($X_n$ is isomorphic to $\left( \oplus_{n=0}^\infty \ell_p(T_\infty^n)\right)_{c_0}$) the separability of $Y^*$ follows.
\end{proof}

\begin{lemma}\label{lemma:cocibidualporespacio}
$Y^{**}/Y$ is reflexive.
\end{lemma}

\begin{proof}
According to Proposition~\ref{prop:interpolcociente} (3) we get that $Y^{**}/Y$ is isometric to a subspace of $Y_0$, which is the one obtained by applying the construction to $X_0:=X^{**}/X$ and $W_0:=Q(\overline{W}^{w^*})$, where $Q:X^{**}\longrightarrow X^{**}/X$ is the canonical quotient operator.

Observe that $\overline{W}^{w^*}\subseteq \frac{K-K}{2^\frac{1}{p}}$, then it is contained in $\overline{\spann}(K)$. Consequently, $Q(\overline{W}^{w^*})$ is contained in $\overline{\spann}(K)/X$. Since $\overline{\spann}(K)/X$ is reflexive (by Theorem~\ref{theo:cocientereflexivo}) we get that $Q(\overline{W}^{w^*})$ is relatively weakly compact. By (5) in Proposition~\ref{prop:intercuentaini} we infer that $X_0$ is reflexive. Since $Y^{**}/Y$ embeds isometrically into $X_0$, we get that $Y^{**}/Y$ is reflexive.
\end{proof}

Now we are ready to providing the pending proof.

\begin{proof}[Proof of Theorem~\ref{theo:espafuerteregular}]
Observe that $Y$ does not contain $\ell_1$ isomorphically since $Y^*$ is separable. Moreover, $Y^{**}/Y$ does not contain $\ell_1$ neither since it is reflexive. Consequently $Y^{**}$ does not contain $\ell_1$ since the property of containing $\ell_1$ is a three-space property. By \cite[Proposition VI.3]{gms} it follows that $B_{Y^{**}}$ (and hence $B_Y$) is strongly regular.
\end{proof}

Call $L:=\frac{3}{\sqrt{3}}W\subseteq B_Y$, where the inclusion follows by Proposition~\ref{prop:intercuentaini} (1). Now we have the following result.

\begin{theorem}\label{theo:abidebigrandesL}
Let $\mathcal U$ be a non-empty relatively weakly open subset of $L$. Then:
\begin{enumerate}
    \item $\mathcal U$ has radius 1.
    \item $\mathcal U$ has diameter $\geq 2^\frac{1}{p}$.
\end{enumerate}
\end{theorem}

\begin{proof}
Consider the natural inclusion $j:Y\longrightarrow X$ on $X$, which is continuous in virtue of Proposition~\ref{prop:intercuentaini} (2). Since $j^{**}$ is injective by (4) of Proposition~\ref{prop:intercuentaini} it follows that $j^*=X^*\longrightarrow Y^*$ has dense range. Hence we can assume without to loss of generality that $j(U)$ defines a non-empty relatively weakly open subset of $\frac{3}{\sqrt{3}}j(W)$. Let us prove from here both assertions.

To begin with, let $y\in Y$ and $\varepsilon>0$. Since $\frac{\sqrt{3}}{3}j(\mathcal U)$ is a non-empty weakly open subset of $W$ then it has radius $1$ by Theorem~\ref{theo:weakopenconjuinter}. Consequently, since $\frac{\sqrt{3}}{3}j(y)\in X$ we can find $\xi\in \frac{\sqrt{3}}{3}j(\mathcal U)$ such that $\left \Vert \frac{\sqrt{3}}{3}j(y)-\xi\right\Vert\geq 1-\varepsilon$. Now we can find $u\in\mathcal U$ such that $\xi=\frac{\sqrt{3}}{3}j(u)$. Puting all together we get
$$1-\varepsilon<\left\Vert \frac{\sqrt{3}}{3}j(y)-\frac{\sqrt{3}}{3}j(u) \right\Vert=\frac{\sqrt{3}}{3}\Vert j(y-u)\Vert\leq \frac{\sqrt{3}}{3}\Vert j\Vert\Vert y-u\Vert=\Vert y-u\Vert,$$
where we have used that $\Vert j\Vert\leq \frac{3}{\sqrt{3}}$ (Proposition~\ref{prop:intercuentaini}, (2)).

Summarising we have proved that, given any $y\in Y$ and any $\varepsilon>0$ there exists $u\in\mathcal U$ such that $\Vert y-u\Vert\geq 1-\varepsilon$. The arbitrariness of $\varepsilon$ concludes (1).

The proof of (2) is quite similar. Since every non-empty weakly open subset of $W$ has diameter $2^\frac{1}{p}$, given $\varepsilon>0$, by the argument above there are $u_1,u_2\in \mathcal U$ such that
\[\begin{split}
2^\frac{1}{p}-\varepsilon<\left\Vert \frac{\sqrt{3}}{3}j(u_1)-\frac{\sqrt{3}}{3}j(u_2) \right\Vert\leq \frac{\sqrt{3}}{3}\Vert j\Vert \Vert u_1-u_2\Vert\leq \Vert u_1-u_2\Vert.
\end{split}\]
The arbitrariness of $\varepsilon>0$ concludes that $\diam(\mathcal U)\geq 2^\frac{1}{p}$ and finishes the proof.
\end{proof}

Indeed, with a similar proof we can even get the following result.

\begin{proposition}\label{prop:improbidual}
Let $\mathcal U$ be a non-empty relatively $w^*$ open subset of $C=\overline{L}^{w^*}\subseteq Y^{**}$. Then $r(\mathcal U)=1$.
\end{proposition}

\begin{proof}
Consider the natural inclusion $j:Y\longrightarrow X$ on $X$, which is continuous in virtue of Proposition~\ref{prop:intercuentaini} (2). First we claim that $j^{**}(C)=\frac{3}{\sqrt{3}}\overline{W^{**}}\subseteq X^{**}$. On the one hand, since $j^{**}$ is $w^*-w^*$ continuous we infer that
$$j^{**}(C)=j^{**}\left(\overline{L}^{w^*} \right)\subseteq \overline{j^{**}(L)}^{w^*}=\overline{\frac{3}{\sqrt{3}} W}^{w^*}=\frac{3}{\sqrt{3}}\overline{W}^{w^*}.$$
For the reverse one, we simply observe that $j^{**}(C)$ is a $w^*$ compact subset of $B_{X^{**}}$ containing $\frac{3}{\sqrt{3}}W$, so it must contain its $w^*$ closure.

Moreover, $j^{**}$ is injective by (4) of Proposition~\ref{prop:intercuentaini} it follows that $j^*=X^*\longrightarrow Y^*$ has dense range, we can assume up to loss of generality that $j^{**}(\mathcal U)$ defines a non-empty relatively $w^*$ open subset of $\frac{3}{\sqrt{3}}\overline{W}^{w^*}$. Let us now prove that $r(\mathcal U)=1$ from here.

In order to do so, let $y^{**}\in Y^{**}$ and $\varepsilon>0$. Since $\frac{\sqrt{3}}{3}j^{**}(\mathcal U)$ is a non-empty $w^*$ open subset of $\overline{W}^{w^*}$ then it has radius $1$ by Theorem~\ref{theo:w*openbidual}. Consequently, since $\frac{\sqrt{3}}{3}j^{**}(y^{**})\in X^{**}$ we can find $\xi^{**}\in \frac{\sqrt{3}}{3}j^{**}(\mathcal U)$ such that $\left \Vert \frac{\sqrt{3}}{3}j^{**}(y^{**})-\xi^{**}\right\Vert\geq 1-\varepsilon$. Now we can find $u^{**}\in\mathcal U$ such that $\xi^{**}=\frac{\sqrt{3}}{3}j^{**}(u^{**})$. Puting all together we get
\[\begin{split}1-\varepsilon<\left\Vert \frac{\sqrt{3}}{3}j^{**}(y^{**})-\frac{\sqrt{3}}{3}j^{**}(u^{**}) \right\Vert& =\frac{\sqrt{3}}{3}\Vert j^{**}(y^{**}-u^{**})\Vert\\
& \leq \frac{\sqrt{3}}{3}\Vert j^{**}\Vert\Vert y^{**}-u^{**}\Vert\\
& =\Vert y^{**}-u^{**}\Vert,
\end{split}\]
where we have used that $\Vert j^{**}\Vert=\Vert j\Vert\leq \frac{3}{\sqrt{3}}$ (Proposition~\ref{prop:intercuentaini}, (2)).

Summarising we have proved that, given any $y^{**}\in Y^{**}$ and any $\varepsilon>0$ there exists $u^{**}\in\mathcal U$ such that $\Vert y^{**}-u^{**}\Vert\geq 1-\varepsilon$. The arbitrariness of $\varepsilon$ concludes (1).
\end{proof}

Now following the ideas behind the proof of \cite[Theorem 3.3 (1)]{blr17} we can now provide the proof of Theorem~\ref{theo:renormainejem}.

\begin{proof}[Proof of Theorem~\ref{theo:renormainejem}.]
Consider the equivalent norm $\vert \cdot \vert$ on $Y$ whose unit ball is 
$$C:=\frac{1}{1+\varepsilon}(L+\varepsilon B_Y).$$
Since $C\subseteq B_Y$ it follows that $\vert y\vert\geq \Vert y\Vert$ holds for every $y\in Y$. 

It is immediate that the bidual unit ball is
$$\overline{C}^{w^*}=\frac{1}{1+\varepsilon}(\overline{L}^{w^{**}}+\varepsilon B_{Y^{**}}).$$
Let us prove that every $w^*$ open subset of the new bidual ball has radius, at least, $\frac{1}{1+\varepsilon}$. In order to do so, select any $y^{**}\in Y^{**}$, any $w^*$ open subset of the bidual unit ball and any $\delta>0$. Up to taking a smaller subset we can assume with no loss of generality that such open set is of the form
$$\frac{1}{1+\varepsilon}(U+\varepsilon V),$$
where $U$ is a non-empty $w^*$ open subset of $\overline{L}^{w^*}$ and $V$ is a non-empty $w^*$ open subset of $B_{Y^{**}}$. Let any $v_0\in V$ and, Proposition~\ref{prop:improbidual}, we can find $u\in U$ such that
$$\Vert ((1+\varepsilon)y^{**}-\varepsilon v_0)-u\Vert\geq 1-\delta.$$
Now
$$1-\delta<  (1+\varepsilon)\left\Vert y^{**}-\frac{1}{1+\varepsilon}(u+\varepsilon v_0)\right\Vert\leq (1+\varepsilon)\left\vert y^{**}-\frac{1}{1+\varepsilon}(u+\varepsilon v_0)\right\vert.$$
Since $\frac{1}{1+\varepsilon}(u+\varepsilon v_0)\in \frac{1}{1+\varepsilon}(U+\varepsilon V)$ and $\delta>0$ was arbitrary we conclude that
$$r\left(\frac{1}{1+\varepsilon}(U+\varepsilon V) \right)\geq \frac{1}{1+\varepsilon},$$
as requested.

The conclusion for the diameter follows by the proof of \cite[Theorem 3.3.1]{blr17}.
\end{proof}

\section{A way to get a possible counterexample}\label{sect:possiblesolut}

As we have pointed out in the Introduction of the present manuscript, we do not know whether the answer to Question~\ref{ques:ejeSRwithrBWOP} is affirmative. We present in this section a possible way to get an affirmative answer for the $w^*$ version in a dual Banach space, that is, a possibility for getting an example of a dual Banach space wich is SR and such that every non-empty relatively $w^*$ open subset of the unit ball has radius $1$.

In order to present a possible line of research we need to introduce a piece of notation. Let $X$ and $Y$ be Banach spaces and let $T:X\longrightarrow Y$ be an injective and bounded operator. According to \cite[p. 544]{gms} say that $T$ is:
\begin{enumerate}
    \item a \textit{semiembedding} if $T(B_X)$ is closed in $Y$;
    \item a \textit{Tauberian embedding} if $T(C)$ is closed in $Y$ for every closed, convex and bounded subset $C$ of $X$ and;
    \item a \textit{$G_\delta$-embedding} if for every closed and bounded subset $F$ of $X$ we have that $\overline{T(F)}\setminus T(F)=\bigcup\limits_{n\in\mathbb N} K_n$, where $K_n\subseteq Y$ is closed for every $n\in\mathbb N$. 
\end{enumerate}

The study of all the above notions were intensively studied in a series of papers \cite{bouros83,gm85,gmscouT,gms}, and many relations of how the RNP, (C)PCP and SR on $Y$ can be transferred to the space $X$ are exhibited. We refer the reader to the above mentioned papers for background on these notions. For background around Tauberian operators we refer to the book \cite{goma10}. 

Let $T:X\longrightarrow Y$ be an injective and bounded operator. In \cite[Proposition II.4]{gms} it is proved that if $Y$ is SR and $T$ is a $G_\delta$-embedding then $X$ is SR. In our next result we will get that the same conclussion holds if we assume on $T$ being Tauberian instead of a $G_\delta$-embedding.

\begin{theorem}\label{theo:taubstrongreg}
Let $X$ and $Y$ be two Banach spaces. Assume that $T:X\longrightarrow Y$ is an injective and Tauberian operator. If $Y$ is strongly regular, then $X$ is strongly regular.    
\end{theorem}

In order to prove it we will prove that the property of being SR is separably determined. Even though we believe this result is well known for speciallist, we include a proof for the sake of completeness. In order to do so, let us recall a bit of notation from \cite{aln2}. Let $Z$ be a subspace of a Banach space $X$.
We say that $Z$ is an \emph{almost isometric ideal} (ai-ideal) in $X$ if
$X$ is locally complemented in $Z$ by almost isometries.
This means that for each $\varepsilon>0$ and for each
finite-dimensional subspace $E\subseteq X$ there exists a linear
operator $T:E\to Z$ satisfying
\begin{enumerate}
\item\label{item:ai-1}
  $T(e)=e$ for each $e\in E\cap Z$, and
\item\label{item:ai-2}
  $(1-\varepsilon) \Vert e \Vert \leq \Vert T(e)\Vert\leq
  (1+\varepsilon) \Vert e \Vert$
  for each $e\in E$,
\end{enumerate}
i.e. $T$ is a $(1+\varepsilon)$ isometry fixing the elements of $E$.
If the $T$ satisfies only (\ref{item:ai-1}) and the right-hand side of
(\ref{item:ai-2}) we get the well-known
concept of $Z$ being an \emph{ideal} in $X$ \cite{gks}.

Note that the Principle of Local Reflexivity means that $X$ is an ai-ideal in $X^{**}$
for every Banach space $X$. Moreover, given an almost isometric ideal $Z$ in a Banach space $X$, it is known that if every convex combination of slices of $B_X$ has diameter $2$ then every convex combination of slices of $B_Z$ has diameter 2 too \cite[Proposition 3.3]{aln2}. The same proof allows to conclude that if every convex combination of slices of $B_X$ has diameter at least $\delta_0>0$ then every convex combination of slices of $B_Z$ has diameter at least $\delta_0$ too.

With this information in mind we can now prove the desired lemma.

\begin{lemma}\label{lemma:srsepardetermined}
Let $X$ be a Banach space. The following assertions are equivalent:
\begin{enumerate}
    \item $X$ is strongly regular.
    \item Every separable subspace of $X$ is strongly regular.
\end{enumerate}
\end{lemma}

\begin{proof}
(1)$\Rightarrow$(2) follows since the property of being strongly regular is hereditary.

(2)$\Rightarrow$(1). Let us assume that $X$ fails to be strongly regular, and let us prove that there exists a separable subspace $Y$ which fails to be strongly regular.

In order to do so, since $X$ fails to be strongly regular we can assume, up to considering an equivalent renorming \cite[Corollary 3.5]{blr17}, that every convex combination of slices of $B_X$ has diameter at least $\delta_0>0$. Now, by an application of \cite[Theorem 1.5]{abrahamsen15} we can find an almost isometric ideal $Z$ in $X$ which is separable. By an adaptation of the argument of \cite[Proposition 3.3]{aln2} we infer that every convex combination of slices of $B_Z$ has diameter, at least, $\delta_0$. Now $Z$ is a separable subspace of $X$ which fails to be strongly regular, as desired.
\end{proof}

Now we are ready to provide the pending proof.

\begin{proof}[Proof of Theorem~\ref{theo:taubstrongreg}]
By Lemma~\ref{lemma:srsepardetermined} it is enough to prove that any separable subspace $Z$ of $X$ is strongly regular. In order to do so, take $Z\subseteq X$ to be a separable subspace. Then $S=T_{|Z}:Z\longrightarrow Y$ is an injective and Tauberian operator so, in particular, it is a semiembedding. By \cite[Proposition 1.8]{bouros83} $S$ is a $G_\delta$ embedding. Since $Y$ is strongly regular and $S:Z\longrightarrow Y$ is a $G_\delta$ embedding we get that $Z$ is strongly regular by \cite[Proposition II.10]{gms}.

To sum up, we have proved that every separable subspace of $X$ is strongly regular, so $X$ is strongly regular in virtue of Lemma~\ref{lemma:srsepardetermined}, as desired.
\end{proof}

We conclude with a possible way to construct a dual Banach space $Z^*$ whose unit ball is strongly regular and all the weak$^*$ open subsets have radius 1. In order to do so, we will follow a construction from \cite[Proposition 2.5]{ll21}.

Let $Y$ and $L$ be satisfying the thesis of Theorem~\ref{theo:maintheorem}. Consider $K:=\overline{L}^{w^*}\subseteq B_{Y^{**}}$. Then $K$ is strongly regular since $Y^{**}$ is strongly regular and every non-empty relatively $w^*$ open subset of $K$ has radius 1. Following the proof of \cite[Proposition 2.5]{ll21}, for every $n\in\mathbb N$ define $\Vert \cdot\Vert_n$ the equivalent (dual) norm on $Y^{**}$ whose unit ball is $K+\frac{1}{n}B_{Y^{**}}$. Now define
$$Z:=\left\{z^{**}\in Y^{**}: \sup_{n\in\mathbb N}\Vert z^{**}\Vert_n<\infty \right\},$$
and define on $Z$ the norm given by $\vert z^{**}\vert=\sup_{n\in\mathbb N}\Vert z^{**}\Vert_n$. From the proof of \cite[Proposition 2.5]{ll21} it follows that $Z=\spann(K)$, that the formal identity $i:Z\longrightarrow Y^{**}$ is continuous and $\Vert i\Vert\leq 1$. Moreover, $(Z,\vert\cdot\vert)$ is a Banach space. Even more, it is a dual Banach space. Indeed, $(Z,\vert\cdot\vert)$ is a dual Banach space whose predual $W$ is described as the norm closure of $\{y^*_{|Z}: y^*\in Y^{*}\subseteq Y^{***}\}$. Consequently, if we consider $i^*:Y^{***}\longrightarrow Z^*$ then $i^*(Y^*)$ is norm-dense in $W$. 

Observe that, $i:Z\longrightarrow Y^{**}$ is a semiembedding since $i(B_Z)=i(K)=K$ is closed. We wonder the following.

\begin{question}\label{quest:tauberian}
Is $i:Z\longrightarrow Y^{**}$ a Tauberian operator?    
\end{question}

If $i$ were a Tauberian operator, then $Z$ would be a dual Banach space which is strongly regular by Theorems~\ref{theo:espafuerteregular} and \ref{theo:taubstrongreg}. Moreover, it follows that every non-empty relatively $w^*$ open subset of the unit ball has radius $1$. Indeed, consider $\mathcal U\subseteq K=B_Z$ we a non-empty relatively $w^*$ open subset, $z^{**}\in Z$ and $\varepsilon>0$, and let us find $u^{**}\in \mathcal U$ such that $\vert z^{**}-u^{**}\vert>1-\varepsilon$. In order to do so, by a density argument, we can assume with no loss of generality that
$$\mathcal U=\left\{u^{**}\in Z: \vert i^*(y^*)(u^{**}-u_0^{**})\vert<\eta, 1\leq i\leq n \right\},$$
for certain $\eta>0$, $y_1^*,\ldots, y_n^*\in Y^*$ and $u_0^{**}\in\mathcal U$. Observe that
$$i(\mathcal U)=\{i(u^{**})\in i(K)=K: \vert y_i^*(i(u^{**})-i(u_0^{**}))\vert<\eta, 1\leq i\leq n\}$$
is a non-empty relatively $w^*$ open subset of $K$. Since every non-empty relatively $w^*$ open subset of $K$ has radius $1$ we can find some $u^{**}\in K$ such that $i(u^{**})\in i(\mathcal U)$ and $\Vert i(u^{**})-i(z^{**})\Vert>1-\varepsilon$. Observe that $u^{**}\in \mathcal U$. Moreover, taking into account that $i$ is a norm-one operator, we obtain that
$$\vert u^{**}-z^{**}\vert\geq \Vert i(u^{**})-i(z^{**})\Vert>1-\varepsilon.$$

The arbitrariness of $\varepsilon>0$ and $z^{**}\in Z$ proves that $r(\mathcal U)=1$, as desired.

\section*{Acknowledgements}  

The authors are deeply grateful to Manuel Gonz\'alez for kindly answering questions around Question~\ref{quest:tauberian}.

This research has been supported  by MCIU/AEI/FEDER/UE\\  Grant PID2021-122126NB-C31, by MICINN (Spain) Grant \\ CEX2020-001105-M (MCIU, AEI) and by Junta de Andaluc\'{\i}a Grant FQM-0185.

\end{document}